\documentclass{article}

\usepackage[all]{xy}
\usepackage[latin1]{inputenc}        
\usepackage{graphicx}
\usepackage{amsfonts,amssymb,amsmath,color,mathrsfs, amstext,bm}
\usepackage{amsbsy, amsopn, amscd, amsxtra, amsthm,authblk, enumerate}
\usepackage{upref}
\usepackage[colorlinks,
            linkcolor=red,
            anchorcolor=red,
            citecolor=red
            ]{hyperref}

\usepackage{geometry}
\usepackage{lineno}
\usepackage{float}

\usepackage{yhmath}

\usepackage{enumitem}

\usepackage[normalem]{ulem}

\usepackage{xcolor}

\usepackage{txfonts}

\numberwithin{equation}{section}

\DeclareMathOperator{\KL}{\mathrm{KL}}
\DeclareMathOperator{\Law}{\mathrm{Law}}

\usepackage{subcaption}

\usepackage{algorithm}
\usepackage{algpseudocode}
\usepackage{caption}

\newtheorem{theorem}{Theorem}[section]
\newtheorem{lemma}{Lemma}[section]
\newtheorem{assumption}{Assumption}[section]
\newtheorem{proposition}{Proposition}[section]

\theoremstyle{remark}
\newtheorem{remark}{Remark}[section]

\newcommand\keywords[1]{\textbf{Keywords:} #1}

\newcommand{\tone}{\textbf{1}}

\newtheorem*{theorem*}{Theorem}

\usepackage{todonotes}

\begin{document}
\title{Long-time reverse transportation inequalities for non-globally-dissipative Langevin dynamics}

\author[a,b,c]{Jianfeng Lu\thanks{E-mail:jianfeng@math.duke.edu}}
\author[a]{Yuliang Wang\thanks{E-mail:yuliang.wang2@duke.edu}}
\affil[a]{Department of Mathematics,  Duke University, Durham, NC, USA.}
\affil[b]{Department of Physics, Duke University, Durham, NC, USA.}
\affil[c]{Department of Chemistry, Duke University, Durham, NC, USA.}

\date{}
\maketitle

\begin{abstract}


    We establish a dimension-free, uniform-in-time reverse transportation inequality for Langevin dynamics with non-convex potentials. This inequality controls the R\'enyi divergence of arbitrary order between the process distributions starting from distinct initial points and serves as the dual version of the Harnack inequality. Notably, we prove that this inequality retains exponential decay in the long-time regime, thereby extending existing results for log-concave sampling to the non-convex setting.
\end{abstract}




\section{Introduction}

Consider overdamped Langevin dynamics for non-log-concave sampling in $\mathbb{R}^d$
\begin{equation}\label{eq:overdamped}
    dX_t = -\nabla U(X_t)dt + \sqrt{2}\,dB_t,
\end{equation}
where $B_t$ is the standard $d$-dimensional Brownian motion and $U: \mathbb{R}^d \rightarrow \mathbb{R}$ is a non-globally convex potential, and the target distribution $\pi \propto e^{-U}$ is the invariant measure of the stochastic differential equation (SDE) \eqref{eq:overdamped}.
For any time $t$, denote the corresponding Markov semigroup by $P_t$, so that $\delta_x P_t$ is the law of $X_t$ that solves \eqref{eq:overdamped} with initial condition $X_0 = x$. In this paper, we establish a long-time reverse transportation inequality (also called the entropy-cost inequality \cite{wang2010harnack, ren2025bi}): 
\begin{equation}\label{eq:KLcontractionintro}
    \KL\left(\delta_x P_T \,\Vert \, \delta_{x'} P_T \right) \leq C \frac{\nu}{e^{\nu T}-1} |x-x'| (|x-x'| \vee \tilde{R}),\quad \forall T > 0.
\end{equation}
Above, $C$, $\nu$ and $\tilde{R}$ are dimension-free positive constants and $\KL\left(\cdot\| \cdot \right)$ denotes the Kullback-Leibler (KL) divergence.
This can also be seen as the dual version of the log-Harnack inequality (see Section \ref{sec:harnack} below).
We further extend the result to the more general R\'enyi divergence, in particular for R\'enyi divergence of order $q \geq 1$ (note that R\'enyi divergence is non-decreasing with respect to the order $q$).  Recall that the R\'enyi divergence of order $q$ between two probability measures $\mu_1 \ll \mu_2$ is defined by
\begin{equation}
    \mathcal{R}_q(\mu_1 \,\Vert\, \mu_2)=
    \begin{cases}
        \displaystyle \frac{1}{q-1} \log \int\left(\frac{d\mu_1}{d\mu_2}\right)^q \mathrm{d} \mu_2 ,                                    & \text{if } 0<q<\infty,\, q \neq 1 ; \\[4mm]
        \displaystyle \int \frac{d\mu_1}{d\mu_2} \log \left(\frac{d\mu_1}{d\mu_2}\right) \mathrm{d} \mu_2 = \KL(\mu_1 \,\Vert\, \mu_2), & \text{if } q=1 .
    \end{cases}
\end{equation}
If $\mu_1$ is not absolutely continuous with respect to $\mu_2$, we set $\mathcal{R}_q(\mu_1 \,\Vert\, \mu_2) = +\infty$. Our proof is mainly based on a specially designed coupling involving an auxiliary process with an additional drift (see for instance \cite{altschuler2024shifted, wang2012coupling}) and a reflection of the noise (see for instance \cite{eberle2016reflection, li2025relative}). Crucially, the strength of the additional drift must be carefully chosen.  See more details in Section \ref{sec:coupling} below.

Establishing inequalities of the form \eqref{eq:KLcontractionintro} is crucial in information theory, optimal transport and Bayesian statistics. It is well known that, via standard transportation inequalities \cite{talagrand1991new, bolley2005weighted}, the Wasserstein distance can be controlled by the square root of KL-divergence when the distribution considered satisfies certain conditions such as a log-Sobolev inequality or a sub-Gaussian tail. The reverse version of this transportation inequality does not generally hold, since the R\'enyi divergence, as a stronger measure of discrepancy between distributions, can blow up when absolute continuity does not hold. However, if we instead consider the divergence between two diffusion systems where the entropy decays in time according to the H-theorem in information theory, the reverse transportation inequality may hold under certain assumptions \cite{wang2010harnack, altschuler2024shifted}, which usually include stronger curvature conditions. Hence, obtaining such reverse transportation inequalities under weaker assumptions is of great significance. On the one hand, we are able to have a better understanding of the Langevin sampling with a non-log-concave target distribution.
On the other hand, as the dual version of the log-Harnack inequality, such results may have broader applications in related fields \cite{ambrosio2015bakry, wang2013harnack, huang2025exponential}.

\paragraph{Literature review.} In recent years, some related results have been established under either weaker metrics or stronger curvature conditions. In \cite{arnaudon2006harnack}, the authors considered auxiliary coupled dynamics with an additional drift and, via the Girsanov theorem, obtained a dimension-free Harnack inequality for diffusion semigroups on Riemannian manifolds. This technique is also called the shifted Girsanov argument in some later papers.
More recently, for the dual version, Altschuler et al. \cite{altschuler2024shifted, altschuler2025shifted} used a similar shifted Girsanov approach for both overdamped and underdamped Langevin dynamics to obtain a long-time reverse transportation inequality which decays exponentially for large time when the potential is globally strongly convex.  In their series, the authors also adapted the method to analyze various numerical schemes for Langevin dynamics.
In \cite{ren2025bi}, using a bi-coupling technique, Ren et al. established a short-time reverse transportation inequality, with an application to McKean-Vlasov SDEs with spatial-distribution dependent noise.
In \cite{huang2025exponential}, Huang et al. proved the exponential ergodicity under both KL divergence and Wasserstein-$2$ distance via hypocoercivity for a class of SDEs with degenerate noise (covering the underdamped Langevin equation with a globally strongly convex potential).
Other related results include the exponential convergence of \eqref{eq:overdamped} and its numerical schemes to the invariant measure $\pi \propto e^{-U}$ in R\'enyi divergence \cite{cao2019exponential, chewi2025analysis, vempala2019rapid, li2022sharp}, where the target need not be log-concave.
To the best of our knowledge, however, there are no rigorous results regarding long-term control for $\mathcal{R}_q\left(\delta_x P_T \,\Vert\, \delta_{x'} P_T \right)$ in the non-convex setting. This gap partially motivates our work.

As mentioned above, in addition to the shifted Girsanov argument \cite{arnaudon2006harnack, altschuler2024shifted, altschuler2025shifted, wang2012coupling}, we employ reflection coupling. This technique, designed to handle non-globally dissipative stochastic systems driven by Brownian motion, first appeared in \cite{lindvall1986coupling}. In recent years, this technique has been adopted and modified to study the Wasserstein contraction / exponential convergence of various stochastic systems, ranging from overdamped and underdamped Langevin dynamics \cite{eberle2016reflection, eberle2011reflection, eberle2019couplings, schuh2024global, kazeykina2024ergodicity, guillin2022convergence}, Hamiltonian Monte Carlo \cite{bou2020coupling, chak2025reflection}, interacting particle systems \cite{eberle2016reflection, guillin2022convergence, schuh2024global, kazeykina2024ergodicity, cao2024empirical, an2025convergence, jin2023ergodicity}, and their numerical discretizations \cite{ye2024error, majka2020nonasymptotic, li2025ergodicity, li2025relative}. See more details in Section \ref{sec:coupling} below.

\paragraph{Main contributions.} We prove uniform-in-time, dimension-free reverse transportation inequalities for overdamped Langevin dynamics with non-convex potentials. Our main contributions are:
\begin{itemize}
    \item We establish entropy-cost inequalities under a one-sided dissipativity condition, covering non-log-concave examples such as double-well and super-quadratically growing potentials.
    \item We obtain explicit long-time exponential contraction rates in KL divergence and R\'enyi divergence between Langevin laws started from different initial points.
    \item We develop a coupling argument that combines the shifted Girsanov method with reflection coupling and a carefully chosen additional drift. This construction also yields the corresponding uniform-in-time Harnack inequalities by duality.
\end{itemize}
The framework is also flexible enough to suggest extensions to discrete-time sampling algorithms and differential privacy for noisy SGD / stochastic gradient Langevin descent with non-convex loss functions.

\paragraph{Organization of the paper.} The remainder of the paper is organized as follows. In Section \ref{sec:coupling}, we give the basic assumptions, the construction of the couplings and the Lyapunov function, as well as some auxiliary lemmas. In Section \ref{sec:KLcontraction}, based on the construction in Section \ref{sec:coupling}, we prove the reverse transportation inequalities under the KL divergence and then extend the result to the R\'enyi divergence. We also prove the duality between the reverse transportation inequalities and uniform-in-time Harnack inequalities in Section \ref{sec:harnack}. Section \ref{sec:conclusion} provides concluding remarks and discusses possible extensions to other models, including the underdamped Langevin equation and stochastic many-body systems without a globally-dissipative drift. More detailed derivations regarding the Girsanov transform and R\'enyi divergence are provided in the Appendix.

\section{Construction of coupling and Lyapunov function}\label{sec:coupling}
\subsection{Basic assumptions}

Throughout this paper, we make the following assumptions for the potential function $U$, which is much weaker than the strong convexity required in many related results. Also note that we do not need a global Lipschitz assumption, instead we assume the weaker one-sided Lipschitz condition as below.

\begin{assumption}\label{ass:farfield}
    There exists $R>0$ and $L \geq m >0$ such that for all $x$, $y \in \mathbb{R}^d$,
    \begin{equation}\label{eq:monotone}
        -(x-y) \cdot \bigl( \nabla U(x) - \nabla U(y)\bigr) \leq \begin{cases}
            -m\lvert x - y \rvert^2, & \text{if} \,\,\lvert x - y \rvert > R,    \\
            L\lvert x - y \rvert^2,  & \text{if} \,\,\lvert x - y \rvert \leq R.
        \end{cases}
    \end{equation}
\end{assumption}
A sufficient condition for Assumption \ref{ass:farfield} is that $U \in C^2$ and is strongly convex outside a compact set.
An example satisfying Assumption \ref{ass:farfield} is $U(x) = |x|^4 - |x|^2$, which is often referred to as double well potential, and classical theories for log-concave sampling cannot be applied to it.

\subsection{Auxiliary coupled dynamics with noise reflection}
Fixing $T>0$ and $x$, $x' \in\mathbb{R}^d$, we construct the following coupled dynamics for $t \in [0,T]$:
\begin{subequations}\label{eq:coupling}
    \begin{align}
             & \quad dX_t = -\nabla U(X_t) dt + \sqrt{2}\,dB_t,          &             & X_0 = x.   \\
         & \left\{
        \begin{aligned}
             & dX''_t = -\nabla U(X''_t) dt +\eta_t\frac{X_t - X''_t}{\sqrt{|X_t - X''_t|}} dt+ \sqrt{2}\,d\bar{B}_t, &  & \text{for}\,\, t < \tau;     \\
             & X''_t = X_t,                                                                                 &  & \text{for} \,\, t \geq \tau, \\
        \end{aligned}
        \right.
         &                                                 & X''_0 = x'.              \\[3mm]
         & \quad dX'_t = -\nabla U(X'_t)dt   + \sqrt{2}\,d\bar{B}_t, &             & X'_0 = x'.
    \end{align}
\end{subequations}
Above, $\tau := \inf\{t\geq 0: X_t = X''_t \}$ and $\eta_t$ is a nonnegative process 
to be determined later (see a deterministic construction in \eqref{eq:etatsingular} for the KL divergence, and a stochastic one in \eqref{eq:boundedeta} for the R\'enyi divergence). A key construction here is that the Brownian motions are reflected, as elaborated below:
Define smooth cut-off functions $rc(r)$ and $sc(r)$ satisfying $\forall r \geq 0$, $rc(r), sc(r) \in [0,1]$, $rc(r)^2 + sc(r)^2 = 1$, and
\begin{equation}\label{eq:rcsc}
    rc(r) = 0 \quad \text{if} \quad r \geq R+1,\quad\,\,rc(r) = 1 \quad \text{if} \quad r \leq R.
\end{equation}
Denote $Z_t := X_t - X''_t$. Let $B_t^{sc}$ and $B_t^{rc}$ be two independent Brownian motions. Then $B_t$, $\bar{B}_t$ are defined by
\begin{subequations}\label{eq:reflectednoise}
    \begin{align}
         & B_t := \int_0^t rc(|Z_s|)\,\mathrm{d}B^{rc}_s + \int_0^t sc(|Z_s|)\,\mathrm{d}B_s^{sc}. \\
         & \notag \\
         & \bar{B}_t :=
        \begin{cases}\displaystyle
            \int_0^t rc(|Z_s|)\left(I - 2\frac{Z_s \otimes Z_s}{|Z_s|^2} \right) \,\mathrm{d}B^{rc}_s+ \int_0^t sc(|Z_s|)\,\mathrm{d}B_s^{sc}, & \quad\text{if} \,\, t<\tau,     \\
            \bar{B}_\tau+ B_t-B_\tau,                                                                                                                             & \quad\text{if}\,\, t \geq \tau. \\
        \end{cases}
    \end{align}
\end{subequations}

The processes $X$, $X'$, $X''$ are well-defined It\^o processes. By L\'evy's characterization, $B_t$ and $\bar{B}_t$ are Brownian motions: the stochastic integrals in \eqref{eq:reflectednoise} have predictable integrands and quadratic variation $tI$, and the continuation after $\tau$ preserves continuity and quadratic variation. Thus $(X, X')$ are two realizations of the overdamped Langevin equation \eqref{eq:overdamped} with initials $(x,x')$. For the interpolation process $X''$ whose initial is $x'$, we will later show that under a suitable choice of $\eta_t$, $X''_t \rightarrow X_t$ in $L^1$ sense as $t \rightarrow T_{-}$. In the singular-control argument, the joint lower-semicontinuity of KL divergence then gives the desired upper bound after passing to the limit; see the remark below for details. See Figure \ref{fig:illustration} for an illustration of the auxiliary dynamics $X''$.
\begin{figure}[htbp]
    \centering
    \includegraphics[width=0.6\linewidth]{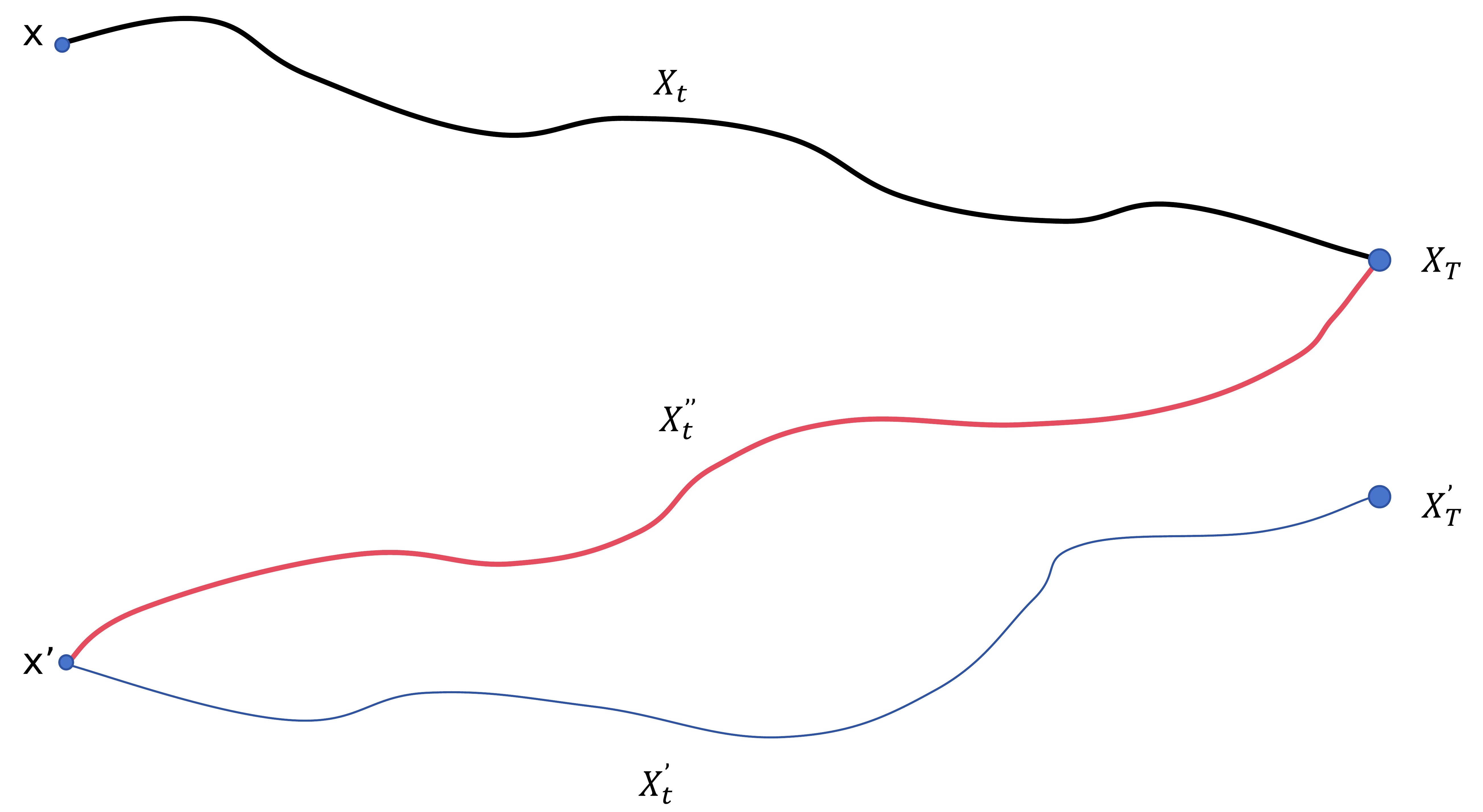}
    \caption{\textbf{An illustration of the coupling \eqref{eq:coupling}.} $X$ and $X'$ are two realizations of the overdamped Langevin equation \eqref{eq:overdamped}. $X''$ is an auxiliary dynamics that interpolates $X$ and $X'$. The KL or R\'enyi divergence between $X'$ and $X''$ is then computed using the Girsanov transform.}
    \label{fig:illustration}
\end{figure}


\begin{remark}
    We choose an $O(\sqrt{|Z_t|})$ additional drift with bounded $\eta_t$ because $L^1$-contraction is well suited to reflection coupling; see, for instance, \cite{eberle2016reflection, schuh2024global, jin2023ergodicity, li2025relative, an2025convergence}. A similar reverse transportation inequality under KL divergence can also be obtained from an $O(|Z_t|)$ additional drift, provided one has an $L^2$ decay estimate for $Z_t$. Such an estimate can be derived from reflection coupling by using a different Lyapunov function $g: \mathbb{R}_{+} \rightarrow \mathbb{R}_{+}$ satisfying $g(r) \leq \tone_{\{r\leq R \}}r + \tone_{\{r>R \}}r^2$ and $g''(r) + rg'(r) \lesssim -g(r)$ for small $r$. Typical choices of $g$ appear in the literature on exponential convergence in Wasserstein-$p$ distance ($p \geq 2$) via reflection coupling; see, for example, \cite{luo2016exponential, majka2020nonasymptotic}. Since this alternative route gives the same upper bound as Theorem \ref{thm:overdamped} up to constants, we do not pursue the details.

    The main advantage of the $O(\sqrt{|Z_t|})$ drift is that it permits a non-singular control $\eta_t$, which is crucial for the R\'enyi-divergence estimate. Indeed, the force $\sqrt{r}$ is stronger than $r$ near zero: for a bounded nonnegative control $\eta_t$, the ODE $\dot{x} = -\eta_t \sqrt{x}$ with $x_0>0$ can reach $0$ in finite time, whereas $\dot{x} = -\eta_t x$ cannot.

\end{remark}

\subsection{Choice of Lyapunov function}
To rigorously characterize the $L^1$-contraction behavior of the coupling above (\textit{i.e.}, exponential decay of $\mathbb{E}|X_t - X''_t|$), a common approach is to seek an increasing concave Lyapunov function $f: \mathbb{R}_{+} \rightarrow \mathbb{R}_{+}$ satisfying
\begin{equation}\label{eq:propertyf}
    rf'(r) + f''(r) \lesssim -f(r) \quad\text{when r is small}.
\end{equation}
For the case of the overdamped Langevin equation, a simple and useful choice of such $f$ is (see also \cite{bou2020coupling, li2025relative, li2025ergodicity}):
\begin{equation}\label{eq:Lyapunov}
    f(r) := \int_0^r e^{-C_f(r' \wedge R_f)}\,\mathrm{d}r',\quad r\geq 0.
\end{equation}
The function $f$ is increasing, concave, and linear for $r \geq R_f$. Moreover, it satisfies
\begin{equation}
    e^{-C_f R_f}r \leq f(r) \leq r,\quad r\geq 0.
\end{equation}
This implies that the Wasserstein-1 distance between two probability measures $\mu_1$ and $\mu_2$ defined by (here $\Pi(\mu_1, \mu_2)$ denotes all joint distributions whose marginal distributions are $\mu_1$ and $\mu_2$, respectively)
\begin{equation}\label{eq:W1}
    W_{1}(\mu_1, \mu_2):=\inf _{\gamma \in \Pi(\mu_1, \mu_2)} \int_{\mathbb{R}^{d} \times \mathbb{R}^{d}}\lvert x - y \rvert \,\mathrm{d} \gamma
\end{equation}
is equivalent to the Kantorovich-Rubinstein distance $W_f$ induced by $f(\cdot)$ defined by
\begin{equation}\label{eq:Wf}
    W_{f}(\mu_1, \mu_2):=\inf _{\gamma \in \Pi(\mu_1, \mu_2)} \int_{\mathbb{R}^{d} \times \mathbb{R}^{d}}f(\lvert x - y \rvert) \,\mathrm{d} \gamma.
\end{equation}
In the proofs below, we will choose
\begin{equation*}
    R_f = R,\quad C_f = \frac{1}{4}R(L+m).
\end{equation*}

\subsection{Some useful lemmas}
We obtain some useful lemmas based on the construction of the coupling and Lyapunov function above. The first observation is that, once $X$ and $X''$ are synchronously coupled in the far field, $|X_t - X_t''|$ has a uniform-in-time almost surely upper bound. 
This fact will be repeated used in the proofs below.

\begin{lemma}\label{lmm:asbound}
    Under Assumption \ref{ass:farfield}, if $\eta_t \geq 0, \; \forall \, t\geq 0$, then it holds that
    \begin{equation}\label{eq:asbound}
        |X_t - X_t''| \leq |x-x'| \vee (R+1) \quad \text{ for }\forall t \geq 0 \text{ a.s.}
    \end{equation}
\end{lemma}
\begin{proof}
    Denote $Z_t = X_t - X''_t$. Since $Z_t = 0$ if $t \geq \tau$, it remains to consider $t < \tau$. By It\^o's formula,  
    \begin{equation}\label{eq:dZt}
        \begin{aligned}
            d|Z_t| & = \frac{Z_t}{|Z_t|} \cdot \left(-\nabla U(X_t) + \nabla U(X_t'') - \eta_t\frac{Z_t}{\sqrt{|Z_t|}} \right) dt + 2\sqrt{2}\,rc(|Z_t|) db^{rc}_t \\
                   & \leq \left(L\tone_{\{|Z_t| \leq R \}} - m\tone_{\{|Z_t| > R\}}\right)|Z_t| dt-\eta_t \sqrt{|Z_t|} dt + 2\sqrt{2}\,rc(|Z_t|) db^{rc}_t,
        \end{aligned}
    \end{equation}
    where $b^{rc}_t$ is a one-dimensional Brownian motion satisfying $db^{rc}_t = \frac{Z_t}{|Z_t|}\cdot dB^{rc}_t$. Since $rc(|Z_t|) = 0$ if $|Z_t| \geq R+1$ by definition, and since $\eta_t$ is always nonnegative, one has
    \begin{equation*}
        d|Z_t| \leq -m|Z_t| dt - \eta_t \sqrt{|Z_t|} \leq 0 \quad \text{if} \quad |Z_t| \geq R+1.
    \end{equation*}
    This gives \eqref{eq:asbound}. A completely rigorous derivation would use the Stroock-Varadhan support theorem \cite{stroock1972support}; we omit it here for simplicity.
\end{proof}

The following proposition is crucial to the proof of Theorem \ref{thm:overdamped}, where we established the decay for $\mathbb{E}|X-X''|$ with the help of the Lyapunov function $f$ defined in \eqref{eq:Lyapunov}. In particular, by choosing a singular control $\eta_t$, we prove that $X$ and $X''$ must meet in finite time. 
We will establish a variant of Proposition \ref{prop:L1contraction} to prove Theorem \ref{thm:renyi} in Section \ref{sec:renyi}, where we will choose a different $\eta_t$ for technical reasons.

\begin{proposition}\label{prop:L1contraction}
    Under Assumption \ref{ass:farfield},  Let $f(\cdot)$ be a Lyapunov function as defined in \eqref{eq:Lyapunov} with $R_f = R$ and $C_f = \frac{1}{4}R(L+m)$. 
    Choose
    \begin{equation}\label{eq:etatsingular}
        \eta_t =\frac{aa'}{e^{a(T-t)}-1},\quad t\in [0,T)
    \end{equation}
    with any $a$, $a'>0$.
    Then for any $t\in [0,T]$ and $x \neq x'$,
    \begin{equation}\label{eq:f_decay}
        \mathbb{E}|X_t-X_t''| \leq C_1\mathbb{E}f(|X_t-X_t''|) \leq C_1\exp\left(-2\nu t - C_1^{-1}\big(\sqrt{|x-x'|}\vee\sqrt{R+1}\big)^{-1}\int_0^t \eta_s ds \right)|x-x'|.
    \end{equation}
    where
    \begin{equation}
        \nu := \frac{m}{2}\exp\bigl(-\frac{1}{4}R^2(L+m)\bigr),\quad C_1 := \exp\bigl(\frac{1}{4}R^2(L+m)\bigr),
    \end{equation}
    In particular,
    \begin{equation}
        X_T = X_T'' \quad a.s..
    \end{equation}
\end{proposition}
\begin{proof}
    Recall that $Z_t := X_t - X_t''$.
    By It\^o's formula, when $t \leq \tau$,
    \begin{equation*}
        df(|Z_t|) = A_tdt + 2\sqrt{2}\,rc(|Z_t|)f'(|Z_t|) db^{rc}_t,
    \end{equation*}
    where $b_t^{rc}$ is a one-dimensional Brownian motion satisfying $db^{rc}_t = \frac{Z_t}{|Z_t|} \cdot dB^{rc}_t$, and
    \begin{equation*}
        A_t :=  f'(|Z_t|)\frac{Z_t}{|Z_t|} \cdot \left(-\nabla U(X_t) + \nabla U(X''_t) - \eta_t \frac{Z_t}{\sqrt{|Z_t|}} \right)+4f''(|Z_t|)rc^2(|Z_t|).
    \end{equation*}
    If $|Z_t| > R$, the far-field convexity condition plays the main role. In this case, $f'' \equiv 0$, $f' \equiv e^{-C_f R}$, and thus
    \begin{equation}\label{eq:boundedby}
        A_t \leq e^{-C_f R}\left(-m|Z_t| - \eta_t \sqrt{|Z_t|}\right).
    \end{equation}
    If $|Z_t| \leq R$, the reflection of noise as well as the concavity of $f$  plays the main role. In this case, since $rc(|Z_t|) = 1$, $A_t$ is bounded by
    \begin{equation*}
        A_t \leq -4C_fe^{-C_f |Z_t|} \frac{|Z_t|}{R} + Le^{-C_f |Z_t|} |Z_t| - \eta_t e^{-C_f|Z_t|}\sqrt{|Z_t|}.
    \end{equation*}
    Since $C_f = \frac{1}{4}R(L+m)$, one obtains \eqref{eq:boundedby} again for the case $|Z_t| \leq R$.
    Moreover, by Lemma \ref{lmm:asbound}, for any $t \geq 0$, $-\sqrt{|Z_t|}\leq -(\sqrt{|x-x'|}\vee\sqrt{R+1})^{-1}|Z_t|$ a.s.. Combining this with \eqref{eq:boundedby} leads to
    \begin{equation*}
        A_t \leq e^{-C_f R}\left(-m - \big(\sqrt{|x-x'|}\vee\sqrt{R+1}\big)^{-1}\eta_t \right)|Z_t|.
    \end{equation*}
    Note that $r \geq f(r)$ for all $r \geq 0$. Hence,
    \begin{equation*}
        \frac{d}{dt}\mathbb{E}[f(|Z_t|)\tone_{\{t<\tau \}}] \leq -e^{-C_f R}\left(m + \big(\sqrt{|x-x'|}\vee\sqrt{R+1}\big)^{-1}\eta_t \right) \mathbb{E}[f(|Z_t|)\tone_{\{t<\tau \}}].
    \end{equation*} 
    Since $f(|Z_t|) = 0$ when $t \geq \tau$, it holds that
\begin{equation*}
        \frac{d}{dt}\mathbb{E}[f(|Z_t|)] \leq -e^{-C_f R}\left(m + \big(\sqrt{|x-x'|}\vee\sqrt{R+1}\big)^{-1}\eta_t \right) \mathbb{E}[f(|Z_t|)].
    \end{equation*} 
Then the first claim \eqref{eq:f_decay} follows from Gr\"onwall's inequality and the fact that $e^{-C_f R}r\leq f(r) \leq r$ for all $r \geq 0$.

Moreover, choosing $\eta_t =\frac{aa'}{e^{a(T-t)}-1}$ with any $a, a'>0$,
\begin{equation*}
    -\int_0^t \eta_s ds = a'\log \frac{e^{a T}-e^{a t}}{e^{a T}-1}.    
\end{equation*}
Then at time $T$,
\begin{equation*}
    \exp\left(-2\nu T - C_1^{-1}\big(\sqrt{|x-x'|}\vee\sqrt{R+1}\big)^{-1}\int_0^T \eta_s ds \right) = e^{-2\nu T} \cdot 0 = 0,
\end{equation*}
which implies $X_T = X''_T$ a.s..
\end{proof}

\begin{remark}\label{rmk:singularity}
The control $\eta_t$ used above is singular as $t\rightarrow T$. To justify the argument rigorously, we may first work on the truncated interval $[0,T-\epsilon]$, where the control is bounded and the process $X''$ is well-defined for every $\epsilon>0$. This is a standard approximation in shifted coupling arguments; see, for instance, \cite[Remark 8]{altschuler2024shifted}. As $\epsilon \rightarrow 0$, we have $\Law(X''_{T-\epsilon}) \rightarrow \Law(X_T)$ and $\Law(X'_{T-\epsilon}) \rightarrow \Law(X'_T)$ weakly. Hence, by the joint lower semicontinuity of the KL divergence,
\begin{equation*}
    \KL(\Law(X_T) \,\Vert\, \Law(X'_T))
    \leq \liminf_{\epsilon \rightarrow 0}
    \KL(\Law(X''_{T-\epsilon}) \,\Vert\, \Law(X'_{T-\epsilon})).
\end{equation*}
It is therefore enough to establish an upper bound on $[0,T-\epsilon]$ uniformly in $\epsilon$ and then pass to the limit.

The same truncation also verifies Novikov's condition for the Girsanov transform. Indeed, with $u_t:=\eta_t(X_t-X_t'')/\sqrt{|X_t-X_t''|}$ before $\tau$ and $u_t=0$ after $\tau$, Lemma \ref{lmm:asbound} gives $|u_t|^2\leq \eta_t^2(|x-x'|\vee(R+1))$ a.s. Since $\eta_t$ is deterministic and $\int_0^{T-\epsilon}\eta_t^2\,dt<\infty$ for each fixed $\epsilon>0$, the relevant exponential moment is bounded by $\exp\{c(|x-x'|\vee(R+1))\int_0^{T-\epsilon}\eta_t^2\,dt\}<\infty$. Thus Novikov's condition holds on $[0,T-\epsilon]$.

\end{remark}

\section{Reverse transportation inequalities}\label{sec:KLcontraction}

\subsection{Reverse transportation inequality under KL divergence}

\begin{theorem}\label{thm:overdamped}
    Under Assumption \ref{ass:farfield}, for any $T>0$ and $x$, $x' \in \mathbb{R}^d$, one has 
    \begin{equation}
        \KL\left(\delta_x P_T \,\Vert\, \delta_{x'} P_T \right) \leq \frac{2C_1^3 \nu }{e^{2\nu T}-1}\,|x-x'|\,\big(|x-x'|\vee(R+1)\big),
    \end{equation}
    where
    \begin{equation}
        \nu := \frac{m}{2}\exp\left(-\frac{1}{4}R^2(L+m)\right),\quad C_1:= \exp\bigl(\frac{1}{4}R^2(L+m)\bigr).
    \end{equation}
\end{theorem}

\begin{remark}
    The rate is likely non-optimal: The choice of the Lyapunov function $f(\cdot)$ (recall \eqref{eq:Lyapunov} above) is not optimized, and it is in fact difficult to define the optimality for the choice of a Lyapunov function. In addition, given a fixed $f(\cdot)$ with free parameters $C_f$, $R_f$, the choice of $\nu$ is not optimized, which is influenced by the choice of $\eta_t$, $C_f$, $R_f$.

    The rate scales like $\frac{1}{T}$ when $T \rightarrow 0$ and $e^{-2\nu T}$ when $T \rightarrow \infty$. Also, $\nu = m/2$ if $U$ is globally strongly-convex. These are all consistent with existing results under the strongly convex setting (see, for instance \cite[Section 4.1]{altschuler2024shifted}). Moreover, under the non-convex setting (Assumption \ref{ass:farfield}), the rate decays exponentially with an $\exp(-R^2)$ coefficient, which is also consistent with related existing results under similar assumptions (see for instance \cite[Section 2.3]{eberle2016reflection}, \cite[Theorem 1.3]{luo2016exponential}, \cite[Theorem 5]{schuh2024global}).

\end{remark}

\begin{remark}[Sharpness of initial distance dependence]
One natural question is whether the initial dependence $|x-x'|(|x-x'| \vee (R+1))$ in Theorem \ref{thm:overdamped} can be improved to the quadratic dependence $|x-x'|^2$. This dependence is closely related to the regularity of $\frac{d^2}{dx^2} P_t$ in $x$, and it is quadratic when $U$ is globally convex. Under the current reflection coupling framework, however, the non-coupling probability of $Z_t = X_t- X''_t$ at a finite time $T$ is of order $O(|x-x'|)$, assuming a bounded drift in the SDE for $Z_t$. In this paper, we overcome this difficulty by either introducing a singular control $\eta_t$ or using a short terminal synchronous coupling. Both approaches yield reverse transportation inequalities (with the first approach restricted to the KL case), but at the cost of losing an $O(|x-x'|)$ factor when $|x-x'|$ is small. Technically, this loss comes from the a.s. bound $|Z_t| \leq |Z_0| \vee (R+1)$ in Lemma \ref{lmm:asbound}.

This open problem is closely related to Wasserstein-$p$ contraction for diffusion processes without a globally dissipative drift, which remains open to the best of our knowledge. For instance, reflection coupling in the literature yields Wasserstein-$1$ contraction, while for Wasserstein-$p$ with $p>1$ it typically yields only exponential decay. One reason is that the Lyapunov functions $g(\cdot)$ used there satisfy $r^p\leq g(r) \lesssim \tone_{\{r \leq R \}}r + \tone_{\{r > R \}}r^p$, rather than the desired $r^p \lesssim g(r) \lesssim r^p$. Existing results for Wasserstein-$p$ contraction with $p\geq 2$ require additional assumptions beyond Assumption \ref{ass:farfield}; see, for instance, \cite{monmarche2026long, monmarche2023wasserstein, huang2025exponential}.

\end{remark}

With the construction of the coupling and Lyapunov function in Section \ref{sec:coupling} above, as well as a careful choice of the time-dependent force strength $\eta_t$ to be presented in the proof, we are able to prove Theorem \ref{thm:overdamped}.


\begin{proof}[Proof of Theorem \ref{thm:overdamped}]
Fix any $T>0$. Without loss of generality, assume $x \neq x'$. Consider the coupling \eqref{eq:coupling} and the choice $\eta_t$ of the form
\begin{equation}
    \eta_t = \frac{aa'}{e^{a(T-t)}-1},
\end{equation}
where $a, a'>0$ are two constants to be determined later.
   By Proposition \ref{prop:L1contraction},
    \begin{equation}\label{eq:meetatT}
       X_T = X''_T \,\, a.s..
    \end{equation}
    Then the Girsanov transform and the data processing inequality give
    \begin{multline}\label{eq:girsanov}
        \KL\left(\delta_x P_T \,\Vert\, \delta_{x'} P_T \right) = \KL\left(\Law(X_T) \,\Vert\, \Law(X'_T) \right)\\
        = \KL\left(\Law(X''_T) \,\Vert\, \Law(X'_T) \right) \leq \frac{1}{4}\int_0^T \eta_t^2\mathbb{E}|X_t - X_t''| \,\mathrm{d}t.
    \end{multline}
By \eqref{eq:f_decay} in Proposition \ref{prop:L1contraction},
    \begin{equation}
    \begin{aligned}
        \KL\left(\delta_x P_T \,\Vert\, \delta_{x'} P_T \right) &\leq \frac{1}{4}C_1|x-x'|\int_0^T \eta_t^2 \exp\left(-2\nu t - C_1^{-1}\big(\sqrt{|x-x'|}\vee\sqrt{R+1}\big)^{-1}\int_0^t \eta_s ds \right)\,\mathrm{d}t\\
        &= \frac{1}{4}C_1|x-x'|\int_0^T\frac{a'^2a^2}{(e^{a(T-t)}-1)^2}e^{-2\nu t}(\frac{e^{aT}-e^{at}}{e^{aT}-1})^{a'C_1^{-1}\big(\sqrt{|x-x'|}\vee\sqrt{R+1}\big)^{-1}} \,\mathrm{d}t
    \end{aligned}
    \end{equation}
Choose 
\begin{equation}
    a' = 2C_1\big(\sqrt{|x-x'|}\vee\sqrt{R+1}\big),\quad a = 2\nu.
\end{equation}
Then
\begin{equation}
    \KL\left(\delta_x P_T \,\Vert\, \delta_{x'} P_T \right) \leq \frac{2C_1^3 \nu }{e^{2\nu T}-1}\,|x-x'|\,\big(|x-x'|\vee(R+1)\big).
\end{equation}
\end{proof}

\subsection{Reverse transportation inequality under R\'enyi divergence}\label{sec:renyi}

Under a similar setting, applying Girsanov transform, for $q > 1$, one can obtain that (see more details in Appendix \ref{app:girsanov}).
\begin{equation}\label{eq:RqGirsanov}
    \mathcal{R}_q(\delta_{x} P_T \,\Vert\, \delta_{x'} P_T) \leq \frac{1}{2(q - 1)}\log \mathbb{E}\left[\exp\left(\frac{(q-1) + 2(q-1)^2}{2} \int_0^T \eta^2_t|X_t - X_t''| \,\mathrm{d}t \right)\right].
\end{equation}
Here the expectation is taken under the controlled coupling, where $X''$ carries the additional drift and satisfies $X_T''=X_T$ a.s. The form of \eqref{eq:RqGirsanov} is chosen so that the coefficient remains finite as $q\to 1$; see Appendix \ref{app:girsanov}. 
Also note that in the construction of this subsection, the control $\eta_t$ (see \eqref{eq:boundedeta} below) depends on the value of $|X_{\delta T} - X''_{\delta T}|$ (for some $\delta\in(0,1)$), which is random. However, this randomness will not invalidate the Girsanov transform. In fact, since the constructed $\eta_t$ below remains zero when $t\in [0,\delta T]$, and since $X''$, $X'$ share the same Brownian motion and initial, we know that $X_t'' = X'_t$ a.s. for any $t \in [0,\delta T]$. This then enables one to apply the Girsanov-type argument combined with the tower property of conditional expectations. See more details in Appendix \ref{app:girsanov}.
Moreover, the choice of $\eta_t$ in this subsection is a.s. bounded without any singularity, so Novikov's condition naturally holds, and there is no need for the $\epsilon$-argument in Remark \ref{rmk:singularity}. 

The estimation for the right-hand side of \eqref{eq:RqGirsanov} is more complicated than the KL divergence, since the expectation is outside the exponential and the time integral. Moreover, unlike the situation of synchronous coupling as in \cite{altschuler2024shifted} and \cite{altschuler2025shifted}, here the randomness of $X- X''$ comes from the trajectory of Brownian motion in $[0,T]$, so the order of expectation cannot be easily changed.

In order to overcome the above challenge, we make use of the a.s. bound in Lemma \ref{lmm:asbound} and propose a non-singular $\eta_t$ (see \eqref{eq:boundedeta} below). These then enable us to control the exponential function with a linear one, namely, for any $\bar{x}>0$, there exists $C_{\bar{x}}>0$ such that
\begin{equation}
    e^x - 1 \leq C_{\bar{x}} x,\quad \forall \,x\in[0,\bar{x}].
\end{equation}
Moreover, in order to retain long-time the exponential decay as well as the finite coupling time (i.e. $Z_T = 0$) properties, we consider a reflection coupling with $\eta_t = 0$ in $[0,\delta T]$ (which provides the contraction), and a synchronous coupling in $[\delta T, T]$ (which forces $Z_t$ to hit $0$ no later than time $T$), where we will choose $1-\delta = O(T^{-1})$ for large $T$.

In detail, for any $\delta \in (0,1)$,  we still consider the same triple $(X,X',X'')$ in \eqref{eq:coupling}, same $rc(\cdot)$, $sc(\cdot)$ in \eqref{eq:rcsc}, but replace \eqref{eq:reflectednoise} with the following:
\begin{subequations}\label{eq:reflectnoise1}
    \begin{align}
         & B_t := \int_0^t rc(|Z_s|)\,\mathrm{d}B^{rc}_s + \int_0^t sc(|Z_s|)\,\mathrm{d}B_s^{sc}. \\
         & \bar{B}_t :=
        \begin{cases}\displaystyle
            \int_0^t rc(|Z_s|)\left(I - 2\frac{Z_s \otimes Z_s}{|Z_s|^2} \right) \,\mathrm{d}B^{rc}_s+ \int_0^t sc(|Z_s|)\,\mathrm{d}B_s^{sc}, & \quad\text{if} \,\, t<\tau \,\, \text{and}\,\,t<\delta T,    \\
            \bar{B}_{\delta T}+B_t - B_{\delta T}, & \quad\text{if} \,\, t<\tau \,\, \text{and}\,\,t\geq \delta T,    \\
            \bar{B}_\tau + B_t - B_\tau,                                                                                                                               & \quad\text{if}\,\, t \geq \tau. \\
        \end{cases}
    \end{align}
\end{subequations}
Compared with \eqref{eq:reflectednoise}, the only difference in \eqref{eq:reflectnoise1} is that we always use the synchronous coupling in the time interval $[\delta T, T]$. Consequently, $X$, $X'$ are still two realizations of \eqref{eq:overdamped} with different initials. More importantly, Lemma \ref{lmm:asbound} is still valid under the modified noise \eqref{eq:reflectnoise1}, because it only relies on the fact that $rc(r) \equiv 0$ for all $r \geq R+1$. Moreover, we derive a similar auxiliary result, which is a modification of Proposition \ref{prop:L1contraction}:

\begin{proposition}\label{prop:L1contraction_new}
Under Assumption \ref{ass:farfield}, consider the coupling structure \eqref{eq:coupling} with \eqref{eq:reflectednoise} replaced by \eqref{eq:reflectnoise1}. Fix any $a\in \mathbb{R}$ and $\delta \in (0,1)$. Choose
 \begin{equation}\label{eq:boundedeta}
    \eta_t = \tone_{\{t \in [\delta T, T]\}} 2\sqrt{|X_{\delta T}-X''_{\delta T}|}\frac{aL-\frac{L}{2}}{e^{(aL-\frac{L}{2})(T-\delta T)}-1}\,e^{aL(t-\delta T)}.
\end{equation} 
We take $\eta_t|_{a=\frac{1}{2}} = \lim_{\tilde{a}\rightarrow \frac{1}{2}}\eta_t|_{a=\tilde{a}}$. Then for any $x \neq x'$,
 \begin{enumerate}
\item For $t\in[0,\delta T]$,
\begin{equation}\label{eq:firstpartL1}
    \mathbb{E}|X_t-X_t''| \leq C_1e^{-2\nu t}|x-x'|.
\end{equation}
The definitions of $C_1$ and $\nu$ above are the same as in Proposition \ref{prop:L1contraction}.
\item For $t \in [\delta T, T]$,
\begin{equation}\label{eq:secondpartpointwise}
    |X_t - X_t''| \leq e^{L(t-\delta T)}|X_{\delta T}-X''_{\delta T}| \left(1 - \frac{e^{(aL-\frac{L}{2})(t-\delta T)}-1}{e^{(aL-\frac{L}{2})(T-\delta T)}-1}\right)^2\,\,a.s..
\end{equation}
Consequently, 
\begin{equation}\label{eq:asequal}
    X_T = X_T''\,\,a.s..
\end{equation}
 \end{enumerate}
\end{proposition}

\begin{proof}
The proof of \eqref{eq:firstpartL1} is exactly the same as that for \eqref{eq:f_decay} in Proposition \ref{prop:L1contraction}, because the coupling structure remains unchanged in the time interval $[0,\delta T]$.
To prove \eqref{eq:secondpartpointwise}, using the synchronous coupling and the one-sided Lipschitz condition for $\nabla U$, one has when $t<\tau$,
\begin{equation*}
    d|Z_t| \leq L|Z_t| dt - \eta_t \sqrt{|Z_t|} dt,
\end{equation*}
which implies 
\begin{equation*}
    \frac{d}{dt}\sqrt{|Z_t|} \leq \frac{L}{2}\sqrt{|Z_t|} - \frac{\eta_t}{2}. 
\end{equation*}
With the current choice of $\eta_t$, Gr\"onwall's inequality and the fact $|Z_t| = 0$ whenever $t\geq \tau$ give \eqref{eq:secondpartpointwise}.
Finally, taking $t=T$ in \eqref{eq:secondpartpointwise} gives \eqref{eq:asequal}.

\end{proof}

Then, combining the pointwise bound in Lemma \ref{lmm:asbound} and the decay of $\mathbb{E}|Z_t|$ in Proposition \ref{prop:L1contraction_new}, we are able to extend Theorem \ref{thm:overdamped} to the more general R\'enyi divergence case. Notably, the method we use to prove Theorem \ref{thm:renyi} can also be applied to prove Theorem \ref{thm:overdamped}. For simplicity, we will not restate the alternative proof for Theorem \ref{thm:overdamped}.

\begin{theorem}\label{thm:renyi}
    Under Assumption \ref{ass:farfield}, recall $\nu$ and $C_1$ in Theorem \ref{thm:overdamped}. For any $q > 1$, $T>0$, $x$, $x' \in \mathbb{R}^d$, and any $a\in\mathbb{R}$, $\delta \in (0,1)$, set $D:=|x-x'|\vee(R+1)$ and $c_q:=\frac{(q-1)+2(q-1)^2}{2}$. Then one has
    \begin{multline}\label{eq:renyi-finalthm}
\mathcal{R}_q (\delta_{x} P_T \,\Vert\, \delta_{x'} P_T)
\le \frac{1}{2(q-1)}\log\Big[1+\frac{C_1|x-x'|\,e^{-2\nu\delta T}}{D}\\
\cdot\Big(\exp\big(4\frac{c_q D^2(a-\tfrac12)^2 L\,(e^{(2a+1)L(T-\delta T)}-1)}{(2a+1)\big(e^{(a-\tfrac12)L(T-\delta T)}-1\big)^2}\big)-1\Big)\Big].
\end{multline}
where the fractions are interpreted by continuous extension when $a=\frac{1}{2}$ or $a=-\frac{1}{2}$. Moreover, with the choice
\begin{equation}
    a=1,\quad 1-\delta = \min(\frac{1}{2},\frac{1}{T}),
\end{equation}
asymptotically,
\begin{equation}\label{eq:asym_thm}
        \mathcal{R}_q(\delta_{x} P_T \,\Vert\, \delta_{x'} P_T) \lesssim
        \begin{cases}
            q\dfrac{D^2}{T}, & \text{as } T \rightarrow 0,      \\[3mm]
            |x-x'|\frac{e^{C_Lc_qD^2}-1}{(q-1)D}e^{-2\nu T},   & \text{as } T \rightarrow \infty,
        \end{cases}
    \end{equation}
    where $C_L:=\frac{L(e^{3L}-1)}{3(e^{L/2}-1)^2}$.
\end{theorem}

\begin{proof}
Fix $T>0$. Without loss of generality, assume $x \neq x'$. Still consider the coupling \eqref{eq:coupling} with \eqref{eq:reflectednoise} replaced by \eqref{eq:reflectnoise1}. Consider $\eta_t$ defined in \eqref{eq:boundedeta} above.
    By Proposition \ref{prop:L1contraction_new}, under this choice of $\eta_t$,
    \begin{equation*}
       X_T = X''_T \,\, a.s..
    \end{equation*}
    Then the Girsanov transform, the data processing inequality, combined with H\"older's inequality imply that
    \begin{multline}\label{eq:renyigirsaanov2}
        \mathcal{R}_q\left(\delta_{x} P_T \,\Vert\, \delta_{x'} P_T \right) = \mathcal{R}_q\left(\Law(X_T) \,\Vert\, \Law(X'_T) \right)\\
        = \mathcal{R}_q\left(\Law(X''_T) \,\Vert\, \Law(X'_T) \right) \leq \frac{1}{2(q - 1)}\log \mathbb{E}\left[\exp\left(c_q \int_{\delta T}^T \eta^2_t|X_t - X_t''| \,\mathrm{d}t \right)\right].
    \end{multline}
where $c_q:=\frac{(q-1)+2(q-1)^2}{2}$. See more details in Appendix \ref{app:girsanov}.
By \eqref{eq:secondpartpointwise} in Proposition \ref{prop:L1contraction_new},
\begin{equation}\label{eq:integral-as-bound}
\int_{\delta T}^T \eta_t^2|X_t-X''_t|\,dt\le |X_{\delta T}-X''_{\delta T}|^2\cdot\Phi(T)\quad\text{a.s.,}
\end{equation}
where the deterministic factor $\Phi(T)$ is defined by
\begin{equation}\label{eq:PhiTdef}
\begin{aligned}
&\quad\Phi(T):=4\frac{(a-\tfrac12)^2 L^2}{\big(e^{(a-\tfrac12)L(T-\delta T)}-1\big)^2}\int_{\delta T}^T e^{(2a+1)L(t-\delta T)}\left(1-\frac{e^{(a-\tfrac12)L(t-\delta T)}-1}{e^{(a-\tfrac12)L(T-\delta T)}-1}\right)^2 dt.
\end{aligned}
\end{equation}
A rough estimate (via $\big(1-\tfrac{e^{(a-\tfrac12)L(t-\delta T)}-1}{e^{(a-\tfrac12)L(T-\delta T)}-1}\big)^2\le 1$) gives
\begin{equation*}
\Phi(T)\le 4\frac{(a-\tfrac12)^2 L\,(e^{(2a+1)L(T-\delta T)}-1)}{(2a+1)\big(e^{(a-\tfrac12)L(T-\delta T)}-1\big)^2}.
\end{equation*}
Combining \eqref{eq:integral-as-bound} with Lemma \ref{lmm:asbound}, which gives $|X_{\delta T}-X''_{\delta T}|\le D$ a.s.~where $D:=|x-x'|\vee(R+1)$, one has
\begin{equation*}
\int_{\delta T}^T \eta_t^2|X_t-X''_t|\,dt\le D\,|X_{\delta T}-X''_{\delta T}|\,\Phi(T)\quad\text{a.s.}
\end{equation*}
Note that for any $\bar{x} > 0$,
\begin{equation*}
    e^x - 1 \leq \frac{e^{\bar{x}}-1}{\bar{x}}x,\,\,\forall\,x\in [0,\bar{x}].
\end{equation*}
Taking $x = c_q D|X_{\delta T}-X''_{\delta T}|\Phi(T)$ and $\bar{x} = c_q D^2\Phi(T)$, one has
\begin{equation*}
\mathbb{E}\exp\!\left(c_q\int_{\delta T}^T\eta_t^2|X_t-X''_t|\,dt\right)\le 1+\frac{e^{c_q D^2\Phi(T)}-1}{D}\,\mathbb{E}|X_{\delta T}-X''_{\delta T}|.
\end{equation*}
Finally, using the $L^1$ contraction (\eqref{eq:firstpartL1} in Proposition \ref{prop:L1contraction_new}), one has
\begin{equation}\label{eq:renyiupperbound_withphi}
\mathcal{R}_q (\delta_{x} P_T \,\Vert\, \delta_{x'} P_T)\le \frac{1}{2(q-1)}\log\left[1+\frac{C_1|x-x'|\,e^{-2\nu\delta T}}{D}\left(e^{c_q D^2\Phi(T)}-1\right)\right].
\end{equation}
Using the upper bound for $\Phi(T)$ above, one finally obtains
\begin{equation*}
\mathcal{R}_q (\delta_{x} P_T \,\Vert\, \delta_{x'} P_T)
\le \frac{1}{2(q-1)}\log\Big[1+\frac{C_1|x-x'|\,e^{-2\nu\delta T}}{D}\\
\cdot\Big(\exp\big(4\frac{c_q D^2(a-\tfrac12)^2 L\,(e^{(2a+1)L(T-\delta T)}-1)}{(2a+1)\big(e^{(a-\tfrac12)L(T-\delta T)}-1\big)^2}\big)-1\Big)\Big].
\end{equation*}
Choose 
\begin{equation}
    a=1,\quad 1-\delta = \min(\frac{1}{2},\frac{1}{T}).
\end{equation}
Asymptotically,
\begin{equation*}
        \mathcal{R}_q(\delta_{x} P_T \,\Vert\, \delta_{x'} P_T) \lesssim
        \begin{cases}
            q\dfrac{D^2}{T}, & \text{as } T \rightarrow 0,      \\[3mm]
            |x-x'|\frac{e^{C_Lc_q D^2}-1}{(q-1)D}e^{-2\nu T},   & \text{as } T \rightarrow \infty.
        \end{cases}
    \end{equation*}
where $C_L:=\frac{L(e^{3L}-1)}{3(e^{L/2}-1)^2}$. The factor $q$ comes from $c_q / (q-1) \leq q$. This proves \eqref{eq:renyi-finalthm} and \eqref{eq:asym_thm}.

\end{proof}

\begin{remark}
Theorem \ref{thm:renyi} has the same scaling as the standard estimate in the globally strongly convex case. When $U$ is globally $m$-strongly convex, one can use a synchronous coupling from time $0$ and repeat the estimate leading to \eqref{eq:renyiupperbound_withphi}.
Under this coupling, the distance scale becomes $D=|x-x'|$, while the prefactor in \eqref{eq:renyiupperbound_withphi} reduces to $1$ since $C_1=1$ and $\delta=0$. Moreover, the deterministic factor in the exponential satisfies $\Phi(T)\lesssim m/(e^{C'mT}-1)$ for some constant $C'>0$. Therefore, there exist constants $C,C'>0$, independent of $T$, $m$, and $q$, such that
\begin{equation}
    \mathcal{R}_q (\delta_{x} P_T \,\Vert\, \delta_{x'} P_T) \leq Cq|x-x'|^2\frac{m}{e^{C'mT}-1}.
\end{equation}
The bound in the present non-convex setting is also consistent with the KL estimate in Theorem \ref{thm:overdamped} in the limit $q \rightarrow 1$. In particular, the dependence on the initial separation reduces to $|x-x'| (|x-x'| \vee (R+1))$ as $q\rightarrow 1$.
\end{remark}

\begin{remark}
The synchronous coupling on $[\delta T,T]$ is essential in the proof of Proposition \ref{prop:L1contraction_new} and Theorem \ref{thm:renyi}. In the proof of Theorem \ref{thm:renyi}, we control an exponential moment by linearizing the exponential function on a bounded interval; this requires a bounded, hence non-singular, control $\eta_t$. The terminal synchronous phase is what allows such a bounded control to force $Z_T=0$ a.s. Indeed, if the reflection coupling \eqref{eq:reflectednoise} were used throughout $[0,T]$, then before the coupling time $\tau$ the distance process would satisfy $d|Z_t|=b_t\,dt+\sigma_t\,dB_t$, where $b_t$ and the non-zero diffusion coefficient $\sigma_t$ are uniformly bounded. The first hitting time of $0$ for such one-dimensional diffusions has a non-trivial tail; see, for instance, \cite{klein1952mean,martin2019long}. Thus one cannot expect $Z_T=0$ a.s. under a purely reflective coupling with bounded control.

The same synchronous phase is also what preserves the exponential decay of $\mathcal{R}_q(\delta_{x} P_T \,\Vert\, \delta_{x'} P_T)$ as $T\rightarrow \infty$. The key input is the almost-sure decay estimate \eqref{eq:secondpartpointwise}. Consequently, after the Girsanov transform, when estimating
$$\frac{1}{2(q - 1)}\log \mathbb{E}\left[\exp\left(c_q \int_0^T \eta^2_t|X_t - X_t''| \,\mathrm{d}t \right)\right],$$
we can first bound the time integral pathwise and only then take expectation. With only an $L^1$ contraction estimate, one would instead have to pass the expectation through the time integral, for instance by Jensen's inequality, which would introduce an additional factor of $T$ in the final exponential term.
\end{remark}

\subsection{Asymptotic uniform Harnack inequalities}\label{sec:harnack}

As mentioned in many related results \cite{altschuler2024shifted, altschuler2025shifted, wang2010harnack}, the reverse transportation inequality obtained in Theorem \ref{thm:overdamped} and Theorem \ref{thm:renyi} above is a dual version of Harnack inequalities. In detail, for any positive smooth test function $\varphi$, for $q=1$, Theorem \ref{thm:overdamped} can be proved to be equivalent to the following log-Harnack inequality:
\begin{equation}\label{eq:logharnack}
    P_T \log \varphi(x) \leqslant C(x, x', T)+\log P_T\,\varphi(x'),
\end{equation}
where
\begin{equation*}
    C(x,x',T) := \frac{2C_1^3 \nu }{e^{2\nu T}-1}\,|x-x'|\,\big(|x-x'|\vee(R+1)\big),
\end{equation*}
and $C_1$ is defined in Proposition \ref{prop:L1contraction}. Moreover, for $q>1$, Theorem \ref{thm:renyi} can be proved to be equivalent to the following power-Harnack inequality:
\begin{equation}\label{eq:powerharnack}
    P_T \varphi(x) \leq C_{q'}(x, x', T) \big(P_T(\varphi^{q'})(x')\big)^{1 / q'}, \quad q' = \frac{q}{q-1},
\end{equation}
where, writing $D:=|x-x'|\vee(R+1)$, the logarithm of the constant satisfies
\begin{equation}
    \log C_{q'}(x, x', T)\lesssim
    \begin{cases}
         (q-1)\dfrac{D^2}{T},& \text{as } T \rightarrow 0;      \\[3mm]
         \dfrac{|x-x'|}{qD}\left(e^{C_Lc_qD^2}-1\right)e^{-2\nu T},& \text{as } T \rightarrow \infty.
    \end{cases}
\end{equation}
Similar results and derivations can also be found in \cite[Section 6]{altschuler2024shifted}. We provide some details below for completeness.

In fact, by Donsker-Varadhan variational principle \cite{boue1998variational, deuschel2001large}, one has
\begin{equation*}
    \int \varphi \,\mathrm{d}(\delta_x P_T) \leq \KL(\delta_x P_T \,\Vert\, \delta_{x'} P_T) + \log \int \exp(\varphi) \,\mathrm{d}(\delta_{x'} P_T)
\end{equation*}
with equality if and only if $\varphi = \log \frac{d(\delta_x P_T)}{d(\delta_{x'} P_T)} + a$, $a \in \mathbb{R}$. The equivalence between Theorem \ref{thm:overdamped} and \eqref{eq:logharnack} follows.

Moreover, note that
$$\big(P_T\left(\varphi^{q'}\right)(x')\big)^{1 / q'}=\|\varphi\|_{L^{q'}\left(\delta_{x'} P_T\right)}$$
and
$$P_T \varphi(x)=\int \varphi \,\mathrm{d}\left(\delta_x P_T\right)=\int \varphi \frac{d\left(\delta_x P_T\right)}{d\left(\delta_{x'} P_T\right)} \,\mathrm{d}\left(\delta_{x'} P_T\right).$$
Then one obtains from H\"older's inequality (with equality if and only if $\varphi = a\left(\frac{d\left(\delta_x P_T\right)}{d\left(\delta_{x'}P_T\right)} \right)^{q-1}$ for some $a > 0$) that the optimal $C_{q'}(x,x',T)$ is given by
\begin{equation*}
    \frac{P_T \varphi(x)}{\big(P_T(\varphi^{q'})(x')\big)^{1 / q'}} \leq \left\|\frac{d\left(\delta_x P_T\right)}{d\left(\delta_{x'} P_T\right)}\right\|_{L^q\left(\delta_{x'} P_T\right)}=\exp \left(\frac{q-1}{q} \mathcal{R}_q\left(\delta_x P_T \,\Vert\, \delta_{x'} P_T\right)\right).
\end{equation*}
This then implies the equivalence between Theorem \ref{thm:renyi} and \eqref{eq:powerharnack}.

We conclude the above equivalent results of asymptotic uniform Harnack inequalities in the following theorem:
\begin{theorem}
    Under Assumption \ref{ass:farfield}, for any $T>0$, $q > 1$ and $x,x' \in \mathbb{R}^d$, the log-Harnack inequality \eqref{eq:logharnack} and the power-Harnack inequality \eqref{eq:powerharnack} hold.
\end{theorem}

\section{Conclusion and Outlook}\label{sec:conclusion}

In this paper, we study Langevin dynamics with a non-convex potential. We prove long-time, dimension-free reverse transportation inequalities under KL divergence and extend them to the more general R\'enyi divergence. We also discuss the duality between these reverse transportation inequalities and Harnack inequalities. The proof combines reflection coupling with a shifted Girsanov argument based on a carefully designed additional drift. This combination may also be useful for other stochastic systems with non-globally-dissipative drifts. We conclude with several possible extensions and open questions.

\paragraph{Extension to underdamped Langevin dynamics.} The first natural question is whether our proof framework can be directly applied to the underdamped Langevin dynamics defined by
\begin{equation}\label{eq:ULD}
    \left\{
    \begin{aligned}
         & dX_t = P_t dt,                                                 \\
         & dP_t = -\nabla U(X_t) dt - \gamma P_t dt + \sqrt{2\gamma}dB_t.
    \end{aligned}
    \right.
\end{equation}
where $\gamma > 0$ is the damping coefficient.
This extension appears possible, but new difficulties arise. To our knowledge, for \eqref{eq:ULD}, both the shifted-Girsanov-type argument \cite{altschuler2025shifted, wang2017hypercontractivity} and the reflection coupling \cite{schuh2024global, kazeykina2024ergodicity, guillin2022convergence, eberle2019couplings} have been well studied separately. The main potential challenge is that, in the region where the potential $U$ is strongly convex, the shifted Girsanov argument requires a time-dependent coordinate twisting (i.e. $(x,p) \mapsto (x, x + \gamma_t^{-1} p)$, $\dot{\gamma}_t \nequiv 0$), which would introduce significant difficulties in constructing the Lyapunov function used for reflection coupling. Besides the current coupling approach, another possible method to obtain reverse transportation inequalities for \eqref{eq:ULD} is to use hypocoercivity with the help of H\'erau's functional \cite{chen2024uniform,villani2009hypocoercivity}.

\paragraph{Extension to mean-field Langevin dynamics.} A promising direction is the extension to mean-field Langevin equations, where the drift term in \eqref{eq:overdamped} also depends on the law of the solution. Such stochastic systems can be viewed as mean-field limits of interacting particle systems and have a wide range of applications. When the noise is nondegenerate, the current framework should be adaptable under suitable assumptions on the nonlinear drift, since reflection coupling extends naturally to this setting. When the noise is degenerate, the extension becomes less direct, for reasons similar to those in the underdamped Langevin dynamics above.

\paragraph{Extension to error analysis for numerical schemes of SDEs.} A more practical application is to extend the current coupling method to the analysis of numerical schemes for SDEs, including Langevin dynamics with non-strongly-convex potentials, thereby yielding effective sampling algorithms in the non-log-concave regime. In particular, following the series \cite{altschuler2024shifted, altschuler2025shifted}, one may develop a KL-based error analysis for both overdamped and underdamped Langevin discretizations by using a similar discrete-time coupling and replacing the Girsanov transform with the so-called ``shifted chain rule" for R\'enyi divergence. Since discrete-time reflection coupling has also been extensively studied \cite{majka2020nonasymptotic}, it is promising to combine these two discrete coupling techniques for KL error analysis. We leave this direction for future work.

\subsection*{Acknowledgments}

The work is supported in part by National Science Foundation via award DMS-2309378. The authors thank Sinho Chewi, Jonathan Mattingly, Panpan Ren, Zhenjie Ren, Feng-Yu Wang and Lihan Wang for helpful discussions.

\appendix

\section{Girsanov transform and R\'enyi divergence}\label{app:girsanov}

We give a rigorous derivation for \eqref{eq:RqGirsanov} (or \eqref{eq:renyigirsaanov2}). To the best of our knowledge, the first such estimate appears in \cite{chewi2025analysis}, and we refer the reader to \cite[Section 6]{chewi2025analysis} for a similar proof.

Fix any $\delta \in (0,1)$. Consider the following SDEs for $t\geq  0$:
\begin{equation*}
    \begin{aligned}
         & dX^{(1)}_t = b^{(1)}(X^{(1)}_t)dt + \sqrt{2}\,dB_t, \\
         & dX^{(2)}_t = \tone_{\{t < \delta T \}}b^{(1)}(X^{(2)}_t)dt + \tone_{\{t \geq \delta T \}}b^{(2)}(X^{(2)}_t)dt + \sqrt{2}\,dB_t, \\
    \end{aligned}
\end{equation*}
with the same Brownian motion $B_t$ and the same initial $X^{(1)}_{0} = X^{(2)}_{0} = x_0 \in \mathbb{R}^d$.
Then a straightforward observation is that 
\begin{equation}\label{eq:sameinitial}
    X_t^{(1)} = X_t^{(2)} \quad a.s.\,\,\forall\,t \in [0,\delta T].
\end{equation}
Viewing $X^{(1)}$ as the controlled interpolation process $X''$ and $X^{(2)}$ as the comparison process $X'$, this is exactly the setting in Section \ref{sec:renyi}. Denote $\mathcal{F}_t := \sigma(B_s, s\leq t)$. Denote the associated path measures by $P^{(i)}_{[0,t]}$ and the time marginal by $P^{(i)}_t$ for $i = 1, 2$ and any $t\geq 0$. By the data-processing inequality \cite{thomas1992information},
\begin{equation*}
    \mathcal{R}_q(P^{(1)}_T \,\Vert\, P^{(2)}_T) \leq \mathcal{R}_q(P^{(1)}_{[0,T]} \,\Vert\, P^{(2)}_{[0,T]}),\quad \mathcal{R}_q(P^{(2)}_T \,\Vert\, P^{(1)}_T) \leq \mathcal{R}_q(P^{(2)}_{[0,T]} \,\Vert\, P^{(1)}_{[0,T]}).
\end{equation*}
Moreover, since $X^{(1)}$ and $X^{(2)}$ share the same drift when $t<\delta T$, by \eqref{eq:sameinitial} and Girsanov theorem \cite{girsanov1960transforming}, conditioning on $\mathcal{F}_{\delta T}$,
\begin{equation}\label{eq:girsanov1}
    \frac{dP^{(1)}_{[0,T]}}{dP^{(2)}_{[0,T]}}(X^{(2)}) = \exp\Big(\frac{1}{\sqrt{2}}\int_{\delta T}^T (b^{(1)} - b^{(2)})(X_s^{(2)})   \,\mathrm{d}B_s
    - \frac{1}{4} \int_{\delta T}^T \left| (b^{(1)} - b^{(2)}) (X_s^{(2)})\right|^2 \,\mathrm{d}s\Big),
\end{equation}
and
\begin{equation}\label{eq:girsanov2}
    \frac{dP^{(2)}_{[0,T]}}{dP^{(1)}_{[0,T]}}(X^{(2)}) = \exp\Big(\frac{1}{\sqrt{2}}\int_{\delta T}^T (b^{(2)} - b^{(1)})(X_s^{(2)})   \,dB_s
    +\frac{1}{4} \int_{\delta T}^T \left| (b^{(1)} - b^{(2)}) (X_s^{(2)})\right|^2 \,\mathrm{d}s\Big).
\end{equation}
The explicit formulas above yield estimates for the R\'enyi divergence in both directions. For the application in Section \ref{sec:renyi}, the post-$\delta T$ drift difference may be a progressively measurable control rather than a state-dependent function; the same argument applies with $(b^{(1)}-b^{(2)})(X_t^{(i)})$ replaced by that adapted control, provided the corresponding Novikov condition holds.

For $\mathcal{R}_q(P^{(1)}_T \,\Vert\, P^{(2)}_T)$, by definition,
\begin{equation*}
    \mathcal{R}_q(P^{(1)}_{[0,T]} \,\Vert\, P^{(2)}_{[0,T]}) = \frac{1}{q-1} \log \mathbb{E}\left[\mathbb{E}\left[\Big|\frac{dP^{(1)}_{[0,T]}}{dP^{(2)}_{[0,T]}}(X^{(2)})\Big|^q \Big| \mathcal{F}_{\delta T}\right]\right] = \frac{1}{q-1}\log \mathbb{E}\left[\mathbb{E}\left[\mathcal{E}_T \Big| \mathcal{F}_{\delta T}\right]\right],
\end{equation*}
where
\begin{equation*}
    \begin{aligned}
        \mathcal{E}_T & := \exp\Big(\frac{1}{\sqrt{2}}\int_{\delta T}^T q(b^{(1)} - b^{(2)})(X_s^{(2)})   \,\mathrm{d}B_s
        - \frac{1}{4} \int_{\delta T}^T q\left| (b^{(1)} - b^{(2)}) (X_s^{(2)})\right|^2 \,\mathrm{d}s\Big) \\
                      & =\exp\Big(\frac{1}{\sqrt{2}}\int_{\delta T}^T q(b^{(1)} - b^{(2)})(X_s^{(2)})   \,\mathrm{d}B_s
        -  \frac{1}{2}\int_{\delta T}^T q^2\left| (b^{(1)} - b^{(2)}) (X_s^{(2)})\right|^2 \,\mathrm{d}s\Big)          \\
                      & \quad \cdot \exp\Big(
        \frac{2q^2-q}{4}\int_{\delta T}^T \left| (b^{(1)} - b^{(2)}) (X_s^{(2)})\right|^2 \,\mathrm{d}s\Big) =: \mathcal{E}^{(1)}_T \cdot \mathcal{E}_T^{(2)}.
    \end{aligned}
\end{equation*}
By H\"older's inequality,
\begin{equation*}
    \mathbb{E}[\mathcal{E}_T |\mathcal{F}_{\delta T}]\leq \sqrt{\mathbb{E}[(\mathcal{E}_T^{(1)})^2|\mathcal{F}_{\delta T}]}\sqrt{\mathbb{E}[(\mathcal{E}_T^{(2)})^2|\mathcal{F}_{\delta T}]}.
\end{equation*}
By It\^o's formula, $\mathbb{E}[(\mathcal{E}_T^{(1)})^2|\mathcal{F}_{\delta T}] = 1$. Therefore, concluding the above, one has
\begin{equation}\label{eq:girsanov11}
\begin{aligned}
    \mathcal{R}_q(P^{(1)}_T \,\Vert\, P^{(2)}_T) &\leq \frac{1}{q-1}\log \mathbb{E}\sqrt{\mathbb{E}[(\mathcal{E}_T^{(2)})^2|\mathcal{F}_{\delta T}]}\\
    &\leq \frac{1}{2(q-1)} \log \mathbb{E}\exp\Big(
    \int_{\delta T}^T \frac{2q^2-q}{2}\left| (b^{(1)} - b^{(2)}) (X_s^{(2)})\right|^2 \,\mathrm{d}s\Big),
\end{aligned}
\end{equation}
where the second inequality is due to Jensen's inequality.

For the application in Section \ref{sec:KLcontraction}, it is more useful to write the same divergence with expectation under $P^{(1)}_{[0,T]}$. Let $r=q-1$. By the definition of R\'enyi divergence,
\begin{equation*}
    \mathcal{R}_q(P^{(1)}_{[0,T]} \,\Vert\, P^{(2)}_{[0,T]})
    =\frac{1}{r}\log \mathbb{E}_{P^{(1)}}\left[\mathbb{E}_{P^{(1)}}\left[\left(\frac{dP^{(1)}_{[0,T]}}{dP^{(2)}_{[0,T]}}\right)^r \Big| \mathcal{F}_{\delta T}\right]\right],
\end{equation*}
Under $P^{(1)}$, writing $B$ for the Brownian motion in the $P^{(1)}$ dynamics, conditioning on $\mathcal{F}_{\delta T}$, this density can be represented as
\begin{equation}\label{eq:RNexpression}
    \frac{dP^{(1)}_{[0,T]}}{dP^{(2)}_{[0,T]}}(X^{(1)})
    =\exp\Big(\frac{1}{\sqrt{2}}\int_{\delta T}^T (b^{(1)} - b^{(2)})(X_s^{(1)})\,\mathrm{d}B_s
    +\frac{1}{4}\int_{\delta T}^T |(b^{(1)}-b^{(2)})(X_s^{(1)})|^2\,\mathrm{d}s\Big),
\end{equation}
The same H\"older argument gives
\begin{equation}\label{eq:Girsanov22}
    \mathcal{R}_q(P^{(1)}_{T} \,\Vert\, P^{(2)}_{T}) \leq \frac{1}{2(q-1)} \log\mathbb{E}_{P^{(1)}}\exp\left(\frac{(q-1) + 2(q-1)^2}{2} \int_{\delta T}^T \left| (b^{(1)} - b^{(2)})(X_s^{(1)})\right|^2 ds\right).
\end{equation}

Compared with \eqref{eq:girsanov11}, \eqref{eq:Girsanov22} does not blow up when $q \to 1$, which makes our final results in Section~\ref{sec:KLcontraction} consistent for all $q$. In the setting of \eqref{eq:coupling}, take $X^{(1)}=X''$, $X^{(2)}=X'$, and replace $(b^{(1)}-b^{(2)})(X_s^{(1)})$ by the adapted control $u_t=\eta_t(X_t-X''_t)/\sqrt{|X_t-X''_t|}$, with $u_t=0$ after the coupling time. Changing measure removes $u_t$ from the dynamics of $X''$, so the comparison process has the same time marginal as $X'$. Since $X_T''=X_T$ a.s., \eqref{eq:Girsanov22} gives
\begin{equation}
    \mathcal{R}_q(\delta_{x} P_T \,\Vert\, \delta_{x'} P_T) \leq \frac{1}{2(q - 1)}\log \mathbb{E}\left[\exp\left(\frac{(q-1) + 2(q-1)^2}{2} \int_{\delta T}^T \eta^2_t|X_t - X_t''| \,\mathrm{d}t \right)\right].
\end{equation}

An argument similar to \eqref{eq:Girsanov22} can be found in \cite[Lemma 11]{zhang2025analysis}. The difference is that, in \cite{zhang2025analysis} the author is using an anticipative Girsanov theorem instead, since the process considered therein is not adapted.

\bibliographystyle{plain}
\bibliography{main}

\end{document}